\documentclass[11pt]{amsart}
\usepackage{amssymb,amsxtra,amsmath,amsfonts}
\usepackage{verbatim,enumerate}
\usepackage{graphicx,color,graphics, bbm} 
\usepackage{xypic}
\usepackage[colorlinks = true,
  linkcolor = blue,
 urlcolor  = blue,
citecolor = blue,
anchorcolor = blue]{hyperref}

\usepackage[margin=2.5cm]{geometry}

\newtheorem{theorem}{Theorem}
\newtheorem{lemma}[theorem]{Lemma}
\newtheorem{proposition}[theorem]{Proposition}
\newtheorem{corollary}[theorem]{Corollary}

\theoremstyle{definition}
\newtheorem{definition}[theorem]{Definition}
\newtheorem{example}[theorem]{Example}

\theoremstyle{remark}
\newtheorem{remark}[theorem]{Remark}

\numberwithin{equation}{section}
\numberwithin{theorem}{section}

\renewcommand{\comment}[1]{}

\newcommand{\KK}{\mathbbm{k}}
\renewcommand{\k}{\mathbbm{k}}
\def\G{\mathbb{G}}
\def\Z{\mathbb{Z}}
\def\Gbb{\mathbb{G}}
\def\Q{\mathbb{Q}}

\def\A{\mathcal{A}}
\def\B{\mathcal{B}}
\def\C{\mathbb{C}}

\def\D{\mathcal{D}}

\def\N{\mathbb{N}}
\def\O{\mathcal{O}}

\def\S{\mathcal{S}}

\def\ZZ{\mathbb{Z}}

\def\K{\mathbbm{k}}

\def\t{\mathbf{t}}

\def\bP{P}

\DeclareMathOperator\End{End}

\DeclareMathOperator\Hom{Hom}

\DeclareMathOperator\ad{ad}

\DeclareMathOperator\gr{gr}

\DeclareMathOperator\Tr{Tr}

\DeclareMathOperator\rank{rank}
\DeclareMathOperator\fun{Fun}
\DeclareMathOperator\Ker{Ker}
\DeclareMathOperator\Res{Res}

\DeclareMathOperator\ind{ind}

\DeclareMathOperator\reg{reg}
\DeclareMathOperator\Id{Id}

\def\vac{{\boldsymbol{1}}}  

\newcommand{\codim}{\operatorname{codim}}

\newcommand{\Spec}{\operatorname{Spec}}
\newcommand{\Lie}{\operatorname{Lie}}

\def\g{\mathfrak{g}}

\def\h{\mathfrak{h}}
\def\n{{\mathfrak{n}}}
\def\sl{\mathfrak{sl}}
\def\lieg{{\mathfrak{g}}}

\def\cc{\mathfrak{c}}
\def\lieh{\mathfrak{h}}
\def\v{\mathfrak{v}}
\def\hg{\widehat{\g}}
\def\z{\mathfrak{z}}
\def\hg{\widehat{\g}}

\def\gl{\mathfrak{gl}}
\def\Sch{\operatorname{Sch}}
\def\SL{\operatorname{SL}}
\def\SO{\operatorname{SO}}
\def\Sp{\operatorname{Sp}}

\def\into{\hookrightarrow}
\def\onto{\twoheadrightarrow}
\def\isoto{\overset{\sim}{\longrightarrow}}

\begin{document}

\title[]{The centre of the modular affine vertex algebra}
\author{Tomoyuki Arakawa, Lewis Topley and Juan J. Villarreal}

\address{Department of Mathematical Sciences,
University of Bath,
Bath, BA2 7AY,
United Kingdom}
\address{Research Institute for Mathematical Sciences, Kyoto University, Yoshidahonmachi, Sakyo Ward, Kyoto, 606-8501, Japan}
\email{arakawa@kurims.kyoto-u.ac.jp, lt803@bath.ac.uk, juanjos3villarreal@gmail.com}


\maketitle
\begin{center}
{\it Dedicated to Victor Kac on the occasion of his 80th birthday.}
\end{center}

\begin{abstract} 
The Feigin--Frenkel theorem states that, over the complex numbers, the centre of the universal affine vertex algebra at the critical level is an infinite rank polynomial algebra. The first author and W.~Wang observed that in positive characteristics, the universal affine vertex algebra contains a large central subalgebra known as the $p$-centre. They conjectured that at the critical level the centre should be generated by the Feigin--Frenkel centre and the $p$-centre. In this paper we prove the conjecture for classical simple Lie algebras for $p$ larger than the Coxeter number, and for exceptional Lie algebras in large characteristics. Finally, we give an example which shows that at non-critical level the center is larger than the $p$-centre.
\end{abstract}


\section{Introduction}

Let $\k$ be an algebraically closed field of positive characteristic $p$. Let $G$ be a connected, simply connected, simple algebraic group scheme such that $p$ larger than the Coxeter number. Under these assumption, $\g = \Lie G$ is a simple Lie algebra, and the normalised Killing form $\kappa$ is non-vanishing on $\g$.

Understanding the geometry of the centre $Z(\g)$ of the enveloping algebra $U(\g)$ is fundamental to describing the representation theory of $\g$. For example the maximal dimension of simple $\g$-modules is equal to the generic rank of $U(\g)$ over $Z(\g)$ \cite[Theorem~6]{Za}, and the smooth locus of $\Spec Z(\g)$ coincides with the Azumaya locus \cite{BG}. These discoveries underpinned the derived localisation theorem \cite{BMR} which relates the derived category of $\g$-modules (with fixed/generalised central characters) to the derived category of coherent sheaves on a Springer fibre. The starting point of such a study is describing $Z(\g)$ algebraically.

On one hand, there is the subalgebra $U(\g)^G \subseteq Z(\g)$ which is isomorphic to the algebra of Weyl group invariants $U(\h)^W$ under the twisted Harish--Chandra homomorphism. We call this the {\it Harish--Chandra centre}. On the other hand, there is a large central subalgebra $Z_p(\g)$ which is specific to the modular setting, known as the {\it $p$-centre}. This is best understood as a filtered lifting of the central subalgebra $\k[\g^*]^p$ of the Poisson algebra $\k[\g^*]$. The spectrum of $Z_p(\g)$ naturally identifies with the Frobenius twist $(\g^*)^{(1)}$. Veldkamp's theorem \cite{V} states that these two subalgebras generate the centre $Z(\g)$ and so
\begin{eqnarray}
\label{e:geometricveldkamp}
\Spec Z(\g) \isoto \h^*/W_\bullet \times_{(\h^*/W)^{(1)}} (\g^*)^{(1)}
\end{eqnarray}
where $W_\bullet$ denotes the $\rho$-shifted action (see \cite[3.1.6]{BMR} for more details). Our goal is to give a similar description of the centre of the universal affine vertex algebra.

The Kac--Moody affinisation $\hg$ of $\g = \Lie(G)$ is the non-split central extension of the loop Lie algebra $\g\otimes \k(\!(t)\!)$ afforded by $\kappa$.
The vacuum module $V^k(\g)$ at level $k$ is the generalised Verma module induced from a maximal parabolic subalgebra of $\hg$, and it carries the structure of a vertex algebra \cite{Bo1}. The level $k = -h^\vee$ is called the {\it critical level}. 
As usual, $h^\vee$ denotes the dual Coxeter number.

For the moment we briefly leave the modular setting to discuss some features of the representation theory of $V^k(\g)$ over the complex numbers. Here there is a well-known dichotomy between the critical and non-critical levels. 
The differences between these cases are quite stark.
The multiplicity of  simple highest weight modules in a Verma module is described  by the Kazhdan-Lusztig polynomials/the inverse Kazhdan-Lusztig polynomials associated with the affine Weyl groups at the non-critical levels \cite{KT},
while at the critical level the multiplicity of  simple hight weight modules in a restricted Verma module \cite{AF1}
is described  by the periodic Kazhdan-Lusztig polynomials \cite{Lus, AF2, FG}.
The primary feature which causes these differences is the nature of the centre $\z V^k(\g)$ \cite{FF}. The centre $\z V^k(\g)$ is trivial away from the critical level (a direct consequence of the conformal grading) whilst at the critical level $\z V^{-h^\vee}(\g)$ is an infinite rank polynomial ring, canonically identified with the algebra of regular functions on the moduli space of certain principle bundles on the formal disk, known as opers \cite{BD}. We remark that the non-negative part of the loop group acts on $V^{-h^\vee}(\g)$ and the centre in the complex setting is equal to $V^{-h^\vee}(\g)^{G[[t]]}$.

Returning to our field $\k$ of positive characteristic, 
the study of vertex algebra in  positive characteristic was initiated in \cite{BR}
and has been developed in e.g. \cite{DR, AW, LM, HJ}.
However,
the representation theory of $V^k(\g)$ is much less established compared to that over $\C$.


In \cite{AW} the first author and Weiqiang Wang introduced the theory of baby Wakimoto modules. One of the basic observations is that the restricted structure on $\hg$ leads to a large central subalgebra $\z_pV^k(\g) \subseteq V^k(\g)$, also dubbed the {\it $p$-centre}. This subalgebra is naturally isomorphic to the coordinate ring $\k[(J_\infty \g^*)^{(1)}]$ on the Frobenius twist of the arc space of $\g^*$.

The natural way to understand a non-commutative algebra is to first understand the semi-classical limit. The Poincar{\'e}--Birkoff--Witt (PBW) filtration of the enveloping algebra $U(\hg)$ gives rise to a filtration of the vertex algebra $V^{k}(\g)$ and the associated graded algebra is the classical affine vertex algebra \cite[\textsection 16.2]{FB}. In Section~\ref{ss:modularPVA} we explain the minor differences involved in the definition of a Poisson vertex algebra whilst working in positive characteristic.

The semiclassical limit $\gr V^k(\g)$ is $G[[t]]$-equivariantly identified with the coordinate ring $\k[J_\infty \g^*]$ of the arc space on $\g^*$. Our first main theorem describes the centre of this Poisson vertex algebra.

\begin{theorem}
\label{T:PVAcentreintro}
Let $G$ be a simply connected, connected simple group scheme over a field $\k$ of characteristic $p$ larger than the Coxeter number of $G$. Then the following holds
\begin{enumerate}
\setlength{\itemsep}{4pt}
\item $\k[J_\infty \g^*]^{G[[t]]}$ polynomial ring on infinitely many generators which we describe explicitly \eqref{eqp2}. The subalgebra $\k[J_\infty \g^*]^p$ of $p$th powers is also central.
\item The centre of $\k[J_\infty \g^*]$ is generated by $\k[J_\infty \g^*]^{G[[t]]}$ and $\k[J_\infty \g^*]^p$.
\item $\k[J_\infty \g^*]^{\g[[t]]}$ is a free $\k[J_\infty \g^*]^p$-module and the basis can be describe explicitly \eqref{e:jetsrestrictedbasis}.
\item $\k[J_\infty \g^*]^{\g[[t]]}$ is isomorphic to the tensor product of $\k[J_\infty \g^*]^p$ and $\k[J_\infty \g^*]^{G[[t]]}$ over the intersection.
\end{enumerate}
\end{theorem}

\begin{remark}
\label{R:weakerhypotheses}
Theorem~\ref{T:PVAcentreintro} can be proven using our methods under the assumption that $p$ is a very good prime for $G$. This means that $p$ is good for $G$ (see \cite[\textsection 6.4]{Jamntrl} and that $p \nmid n+1$ in type {\sf A}$_n$. We have worked under the assumption $p > h$ for simplicity, in order to maintain uniform hypotheses throughout the paper.
\end{remark}

In the final section of this paper we deduce a description of the centre of the vertex algebra $V^{-h^\vee}(\g)$ at the critical level. 
\begin{theorem}
\label{T:maintheorem}
Suppose that we are in one of the following two situations:
\begin{enumerate}
\item[(i)] $\g$ is a classical simple Lie algebra, and $p$ larger than the Coxeter number,
\item[(ii)] $\g$ is an exceptional simple Lie algebra and $p \gg 0$.
\end{enumerate}
Then we have:
\begin{enumerate}
\setlength{\itemsep}{2pt}
\item $V^{-h^\vee}(\g)^{G[[t]]}$ is a free commutative differential algebra generated by $r$ elements.
\item $\z V^{-h^\vee}(\g)$ is generated by $V^{-h^\vee}(\g)^{G[[t]]}$ and $\z_p V^{-h^\vee}(\g)$ as a commutative algebra.
\item $\z V^{-h^\vee}(\g)$ is a free $\z_p V^{-h^\vee}(\g)$-module of finite rank over $\z_p V^{-h^\vee}(\g)$.
\item $\z V^{-h^\vee}(\g)$ is isomorphic to the tensor product of $\z_p V^{-h^\vee}(\g)$ and $V^{-h^\vee}(\g)^{G[[t]]}$ over their intersection.
\end{enumerate}
\end{theorem}

Let $\mathcal{X}^{-h^\vee} := \Spec V^{-h^\vee}(\g)^{G[[t]]}$. From Theorem~\ref{T:maintheorem} we deduce a geometric description of the centre
\begin{eqnarray}
\label{e:geometricvertexveldkamp}
\Spec \z V^{-h^\vee}(\g) \isoto \mathcal{X}^{-h^\vee} \times_{(J_\infty \h^*/W)^{(1)}} (J_\infty \g^*)^{(1)}
\end{eqnarray}
The map $(J_\infty \g^*)^{(1)} \to (J_\infty \h^*/W)^{(1)}$ is the Frobenius twist of the jet morphism of the adjoint quotient $\g^* \to \h^*/\!/W$, whilst $\mathcal{X}^{-h^\vee} \to (J_\infty \h^*/W)^{(1)}$ arises from the natural maps $\k[(J_\infty \g^*)^{(1)}]^{G[[t]]} \cong \z_p V^{-h^\vee}(\g)^{G[[t]]} \to V^{-h^\vee}(\g)$. This should be compared with \eqref{e:geometricveldkamp}.

Theorem~\ref{T:maintheorem} is deduced from Theorem~\ref{T:PVAcentreintro} using a process of reduction modulo $p$. We make use of explicit formulas for the generators of $V^{-h^\vee}(\g)$ discovered over $\C$ by Molev \cite{MoSS, Mo}, in classical types. For exceptional types we use a rather primitive form of modular reduction, depending on the Feigin-Frenkel theorem \cite{FF}. 

Now in \cite{AW}, the authors also conjectured that the centre of $V^{k}(\lieg)$ for $k\neq -h^{\vee}$ is given only by the $p$-centre. We have the following counterexample 

\begin{example} Let $S:=\frac{1}{2}(\frac{1}{2}h_{-1}h_{-1}+f_{-1}e_{-1}+e_{-1}f_{-1})\vac\in V^{k}(\mathfrak{sl}_2)$ for arbitrary $k\in \k$. For $p=3$ the following element is central\footnote{In particular, from a direct calculation for $a\in \{e,f,h\}$ mod $3$ we have \begin{align*}
a(z)(2T(:S(w)S'(w):)+2(k+2)T^{(4)}S(w))
\sim\frac{2(k+2)}{(z-w)^3} T(:a(w)S(w):)+\frac{2(k+2)}{(z-w)^2}T^2(:a(w)S(w):)\, ,
\end{align*}
where $T$ denotes the translation operator, see Definition \ref{d2.3}, and $(\cdot)'=T(\cdot)$. Additionally, we have that $a(z)(c(w)-2T(:S(w)S'(w):)-2(k+2)T^{(4)}S(w))$ satisfies the same relation but with factor  $(k+2)$ instead of  $2(k+2)$. } 
in $V^{k}(\mathfrak{sl}_2)$
\begin{align*}
c:=& (S_{(-1)}S_{(-3)}-S_{(-2)}S_{(-2)})\vac+2(k+2)S_{(-5)}\vac+(k+2)k (h_{-6}-f_{-3}e_{-3}  -h_{-3}^2)\vac +\\
&(k+2)( f_{-2}e_{-1}h_{-3} +f_{-1}e_{-3}h_{-2} +f_{-3}e_{-2}h_{-1}- f_{-1} e_{-2} h_{-3} -f_{-3}e_{-1}h_{-2} - f_{-2}e_{-3}h_{-1}  )\vac\, .
\end{align*}
Note that the top graded component of $c$ with respect to the PBW filtration is given by $\gr (c)=2SS''-S'S'$. Hence $\gr (c) \in \k[J_\infty \mathfrak{sl}_2^*]^{SL_2[[t]]}$ and $\gr (c)=\notin \k[J_\infty \mathfrak{sl}_2^*]^p$, see Theorem \ref{T:PVAcentreintro}.
\end{example}

We note that this example implies that \cite[Conjecture 6.8]{JLM} is false: one can check that central element given here does not lie in ideal of $V^k(\sl_2)$ generated by the augmentation ideal of the $p$-centre. Thus the restricted quotient cannot be simple.

We briefly describe the structure of the paper. Section~\ref{S:groupsand modaffine} surveys the required theory of group schemes, arcs and jets, and sets up our conventions for defining the Kac-Moody affinisation. We also recall the precise statement of Veldkamp's theorem and its semiclassical limit.

In Section~\ref{S:modularVOAs} we recall the definition of a Poisson vertex algebra and vertex algebra over $\k$, paying special attention to the minor differences to the complex setting, which is better understood. We explain that the vertex algebra $V^k(\g)$ naturall degenerates to $\k[J_\infty\g^*]$ and we construct the basic $G[[t]]$-invariants which freely generate $\k[J_\infty\g^*]^{G[[t]]}$. Next we outline the procedure of modular reduction for invariants in vertex algebras. Finally we prove Theorem~\ref{T:PVAcentreintro} following an approach analogous to \cite{Fr}, adapting to the modular setting using a powerful theorem of Skryabin \cite{S}.

In the final Section~\ref{S:criticalcentre} we prove Theorem~\ref{T:maintheorem} using the tools of modular reduction and Molev's formulas \cite{MoSS, Mo}.

We expect that Theorem~\ref{T:maintheorem} will hold under the assumption that $p$ is very good for $G$. We note that the methods this paper are already sufficient to prove the theorem in type {\sf G}$_2$ for $p > 3$: Theorem~\ref{T:PVAcentreintro} holds under this hypothesis (Remark~\ref{R:weakerhypotheses}) whilst the formulas for Segal--Sugawara vectors discovered by Yakimova in \cite[\textsection 6]{Ya} can be defined over a $\Z[6^{-1}]$-lattice in $V^{-h^\vee}(\g)$, and so the Theorem follows in this case, following the observations of Section~\ref{ss:completingBCD} verbatim. Similar remarks hold for $\sl_N$ when $p \nmid N$, using the formulas of \cite{MC}, and the remarks of Section~\ref{ss:completingA}.

\subsection*{Acknowledgements}  We would like to thank Dmitriy Rumynin and Alexander Molev for useful correspondence on the subject of this paper, Andrew Linshaw for sharing his knowledge of the OPEdefs 3.1 Mathematica package. We are grateful to University of Bath, RIMS, Ningbo University and SUS Tech international center for mathematics for the excellent working conditions were this work was done. We offer special thanks to Gurbir Dhillon for numerous enlightening conversations. The first author is partially supported by JSPS KAKENHI Grant Number J21H04993. The second and third author are supported by a UKRI Future Leaders Fellowship, grant numbers MR/S032657/1, MR/S032657/2, MR/S032657/3.



\section{Group schemes, jets and arcs}
\label{S:groupsand modaffine}
\subsection{Arcs and jets}
\label{ss:arcsandjets}

Throughout this paper we fix an algebraically closed field $\k$ of characteristic $p > 0$. All vector spaces, algebras and schemes will be defined over $\k$, with the exception of Section~\ref{ss:integralforms} where we discuss modular reduction.

Here we discuss jet schemes and arc spaces. We refer the reader to \cite{Ish} for slightly more detail.

A {\it higher derivation} of a commutative algebra $A$ is a tuple $\partial = (\partial^{(i)} \mid i\ge 0)$ of linear endomorphisms of $A$ with $\partial^{(0)}$ equal to the identity, satisfying the following relations for $i \ge 0$:
\begin{eqnarray}
\label{e:HSrel}
\begin{array}{l}
\partial^{(i)} (ab) = \sum_{i_1 + i_2 = i} \partial^{(i_1)}(a) \partial^{(i_2)}(b).
\end{array}
\end{eqnarray}
Similarly a {\it higher derivation of order $m$} is a tuple $\partial = (\partial^{(i)} \mid i=0.,,,.m)$ satisfying \eqref{e:HSrel} for $i, j \le m$. A differential algebra (of order $m$) $(A, \partial)$ is an algebra equipped with a higher derivation (of order $m$). 

There is a universal differential algebra $(J_\infty A, \partial)$ with a homomorphism $\alpha : A\to J_\infty A$ such that, for every other differential algebra $(B, \sigma)$ admitting a homomorphism $\beta : A \to B$, there is a unique homomorphism $\varphi$ of differential algebras such that $\varphi \circ \alpha = \beta$.

We construct $J_\infty A$ as follows. First assume $A = \k[x_k \mid k \in K]$  is a polynomial ring, where $K$ is an index set. Define $J_\infty A := \k[x_{k}^{(j)} \mid k\in K, \ 0\le j]$. The homomorphism $\alpha : A \to J_\infty A$ is given by $x_k \mapsto x_k^{(0)}$ and the differential $\partial$ is determined uniquely by \eqref{e:HSrel} along with $\partial^{(i)} x_{k}^{(0)} = x_k^{(i)}$ and 
$$\partial^{(i)} \partial^{(j)} = \binom{i+j}{i} \partial^{(i+j)}$$ for all $i,j,k$ in the appropriate ranges. If $I \subseteq A$ is an ideal then we can define $J_\infty I := (\partial^{(i)} a \mid 0\le i, \ a\in I)$, which is a differential ideal of $J_\infty A$. 

Now if $A$ has a presentation $A = \k[x_k \mid k\in K] / I$ then we define $J_\infty A := J_\infty \k[x_k \mid k \in K] / J_\infty I$. The map $A \to J_\infty A$ and higher derivation $\partial$ are defined as above, and $A \to J_\infty A$ is readily seen to inject. It is not hard to see that the isomorphism type of $J_\infty A$ is independent the presentation, and that $J_\infty A$ satisfies the universal property mentioned above. It follows immediately that $J_\infty$ is an endofunctor on the category of commutative algebras, and $J_\infty A$ is known as the arc algebra of $A$.

If $A$ is a Hopf algebra with comultiplication $\Delta : A \to A\otimes A$, counit $\eta : A \to \k$ and antipode $S: A \to A$, then there is a natural Hopf structure on $J_\infty A$ extending the structure on $A$. The existence of such a structure follows from functoriality of $J_\infty$, however we can describe it explicitly: for $a\in A \subseteq J_\infty A$ we have
\begin{eqnarray}
\label{e:JHopfstructure}
\begin{array}{l}
 (\Delta(\partial^{(i)}a)) = \partial^{(i)}\Delta(a);\\
 S(\partial^{(i)}(a)) = \partial^{(i)} S(a);\\
 \eta(a) = a + J_m \Ker \eta.
 \end{array}
\end{eqnarray}
where we equip $A\otimes A$ with the structure of a differential algebra in the obvious manner, and note that $J_\infty A / J_\infty \Ker \eta \cong  J_\infty \k = \k$. 

For $m \in \N$ there is an analogous construction of a universal differential algebra $J_m A$ of order $m$. Briefly, $J_m \k[x_k \mid k \in K]$ is the subalgebra of $J_\infty \k[x_k \mid k\in K]$ generated by $\{x_k^{(j)} \mid k\in K, 0 \le j \le m\}$ with differential defined as before, along with the condition $\partial^{(i)} x_k^{(j)} = 0$ for $i + j > m$. We have $J_\infty A = \varinjlim J_mA$ as algebras.

If $X$ is an affine scheme then the {\it $m$-jet scheme} $J_m X$ is defined to be $J_m X = \Spec J_m A$ and the {\it arc space} is $\Spec J_\infty X$. These functors behave well under gluing and thus $J_m$ is a functor from schemes to schemes.  If $X$ has finite type then so do $J_m X$ for all $m \in \N$, whilst $J_m X$ is of profinite type.

Jet schemes and arc spaces are uniquely determined by the adjunctions
\begin{eqnarray*}
& & \Hom_{\Sch /\k}(Y, J_m X) = \Hom_{\Sch /\k}(Y \times_{\Spec(\k)} \Spec(\k[t]/(t^{m+1})), X),\\
& & \Hom_{\Sch /\k}(Y, J_\infty X) = \Hom_{\Sch /\k}(Y \times_{\Spec(\k)} \Spec(\k[[t]], X).
\end{eqnarray*}
and so the $\k$-points of $J_m X$ are naturally identified with the $\k[t]/(t^{m+1})$-points of $X$ for $m \in \N$, and with the $\k[[t]]$-points of $X$ for $m = \infty$.

If $G$ is a group scheme then, by our above remarks, both $J_m G$ and $J_\infty G$ are group schemes.
\begin{lemma}
\label{L:generateslemma}
Suppose that $G$ is a group scheme which is generated by a collection $(H_k)_{k\in K}$ of subgroup schemes. Then $J_\infty G$ is generated by $(J_\infty H_k)_{k \in K}$. 
\end{lemma}
\begin{proof}
Let $A$ be a commutative Hopf algebra. Observe that the collection of Hopf ideals of $A$ is stable under sums, but not intersections. It follows that for any algebra ideal $I \subseteq A$ there is a unique maximal Hopf ideal subject to $H(I) \subseteq I$. More precisely it is sum of all Hopf ideals contained in $I$. 

We claim that for any ideal $I$ we have $J_\infty H(I) = H(J_\infty I)$. Using properties \eqref{e:JHopfstructure} one can show that $J_\infty H(I)$ is a Hopf ideal of $J_\infty A$ contained in $J_\infty I$, and so $J_\infty H(I) \subseteq H(J_\infty I)$. Conversely we have $H(I) \subseteq I \subseteq J_\infty I$ and so $H(I) \subseteq HJ_\infty(I)$. In order to obtain $J_\infty H(I) \subseteq H J_\infty(I)$ it suffices to show that $H J_\infty(I)$ is closed under $\partial$. Using \eqref{e:JHopfstructure} once again one can easily show that $H$ preserves the class of  differential-closed ideals, and this completes the proof of the claim.

For $k \in K$ we let $I_k \subseteq \k[G]$ be the Hopf ideal defining the subgroup $H_k$. The hypothesis of the lemma is equivalent to the assertion $H(\cap_k I_k) = 0$, and the goal is to show that $H(\cap_k J_\infty I_k) = 0$. Now we have $H(\cap_k J_\infty I_k) = H(J_\infty \cap_k I_k) = J_\infty H(\cap_k I_k) = 0$, and the proof is complete.
\end{proof}

\subsection{Simple group schemes}
\label{ss:simplegroupschemes}

Let $G$ be a connected, simple group scheme of rank $r$, and write $\g = \Lie(G)$. The Coxeter number (resp. dual Coxeter number) is denoted $h$ (resp. $h^\vee$).

For the rest of the paper we will assume that $p > h$ which has the following consequences:
\begin{enumerate}
\setlength{\itemsep}{2pt}
\item[(i)] $2h^\vee$ is invertible;
\item[(ii)] The Lie algebra $\g$ is simple (see \cite[Proposition~1.2]{Ka70} for example);
\item[(iii)] the space of $\g$-invariant forms on $\g$ is one dimensional and every nonzero $\g$-invariant form is non-degenerate, by (ii);
\item[(vi)] The Killing form is non-vanishing, hence non-degenerate (see \cite[Theorem~4.8]{SS}).
\end{enumerate}

Throughout this article we fix the {\it normalised Killing form} on $\g$ which is defined by
\begin{eqnarray}
\label{e:normalisedKilling}
\kappa(x, y) = \frac{1}{2h^\vee} \Tr(\ad(x) \ad(y)).
\end{eqnarray}
When we consider the induced bilinear form on $\h^*$, the normalisation ensures that the square length of every long root is equal to 2, which is the standard convention in the literature \cite{Fr, Ka}. In classical types this form is just the trace form $x, y\mapsto \Tr(xy)$ associated to the natural representation.

We fix a torus $T \subseteq G$, let $X^*(T)$ denote the character lattice, and write $\Phi \subseteq X^*(T)$ for the roots. For $\alpha \in \Phi$ we consider the morphism of group schemes corresponding to the $\alpha$-root subgroup $u_\alpha : \Gbb_a \to G$ (see \cite[\textsection 21.c]{Mi}).

\begin{lemma} 
\label{L:Gschemegenerators}
\begin{enumerate}
\item $G$ is generated by $\{u_\alpha \mid \alpha\in \Phi\}$.
\item $J_\infty G$ is generated by $\{J_\infty u_\alpha \mid \alpha\in \Phi\}$.
\end{enumerate}
\end{lemma}
\begin{proof}
Part (1) is \cite[Proposition~21.62]{Mi}, whilst part (2) follows from Lemma~\ref{L:generateslemma}.
\end{proof}

\subsection{Chevalley restriction and Veldkamp's theorem}

 Let $U(\g)$ be the universal enveloping algebra of $\g$, equipped with the  Poincar{\'e}--Birhoff--Witt  (PBW) filtration $U(\g) = \bigcup_{i\ge 0} U(\g)_i$. The associated graded algebra is $\gr U(\g) = S(\g) = \k[\g^*]$ as $G$-algebras.
 
Write $Z(\g)$ for the centre of $U(\g)$. Both $G$ and $\g$ act naturally on $U(\g)$ and $\k[\g^*]$ via automorphisms extending the adjoint representation, and we have the following obvious inclusions $U(\g)^G \subseteq U(\g)^\g = Z(\g)$ and $\k[\g^*]^G \subseteq \k[\g^*]^\g$, as well as $\gr U(\g)^G \subseteq \k[\g^*]^G$ and $\gr Z(\g) \subseteq \k[\g^*]^\g$.

We write $\k[\g^*]^p = \{f^p \mid f\in \k[\g^*]\}$ for the subalgebra of $p$th powers. There is an obvious inclusion $\k[\g^*]^p \subseteq \k[\g^*]^\g$ since $\g$ acts by derivations.

Recall that $\g$ admits a natural $G$-equivariant restricted structure $x \mapsto x^{[p]}$, which gives rise to a large central subalgebra of $U(\g)$ known as the {\it $p$-centre}. It is generated by elements $x^p - x^{[p]}$. Since the map $x\mapsto x^p - x^{[p]}$ is $p$-semilinear $Z_p(\g)$ is a polynomial ring in $\dim(\g)$ generators, and its spectrum is naturally identified as a $G$-set with the first Frobenius twist of the coadjoint representation $\Spec Z_p(\g) = (\g^*)^{(1)}$, and $\gr Z_p(\g) = \k[\g^*]^p$ under  the identification $\gr U(\g) = \k[\g^*]$; see \cite{Jamntrl} for more detail.

The next result, known as Veldkamp's theorem, describes the centre of $U(\g)$ for $p > h$; \cite{V}. Numerous authors improved the bound on $p$ (see \cite[\textsection 9]{Jamntrl}), and in fact the theorem holds under the standard hypotheses \cite[Theorem~3.5]{BG}. 

The statement presented here can easily be deduced from \cite[Section~3]{BG} and \cite[Theorem~5.4(1)]{S}, see also \cite{KW}.


\begin{theorem}\label{th1}
We have the following:
\begin{enumerate}
\item $\k[\g^*]^G$ is generated by $r$ algebraically independent elements $\bP_1,...,\bP_r$.
\item $\k[\g^*]^\g$ is generated by $\k[\g^*]^G$ and $\k[\g^*]^p$
\item $\k[\g^*]^\g$ is a free $\k[\g^*]^p$-module with basis given by restricted monomials
$$\{\bP_1^{k_1} \cdots \bP^{k_r}_r \mid 0 \le k_j < p \text{ for all } j\}.$$
\item $\k[\g^*]^\g \cong \k[\g^*]^G \otimes_{(\k[\g^*]^p)^G} \k[\g^*]^p$. 
\item The $\bP_i$ admit filtered lifts $Z_1,...,Z_r$ and $U(\g)^G$ is a polynomial algebra on $r$ generators.
\item $Z(\g)$ is generated by $U(\g)^G$ and $Z_p(\g)$.
\item $Z(\g)$ is a free $Z_p(\g)$-module of rank $p^{\dim \g}$, with basis given by restricted monomials
$$\{Z_1^{i_1}\cdots Z_r^{i_r} \mid 0 \le i_j < p \text{ for all } j\}.$$
\item $Z(\g) \cong U(\g)^G \otimes_{Z_p(\g)^G} Z_p(\g)$. $\hfill\qed$
\end{enumerate}
\end{theorem}

\section{Modular vertex algebras}
\label{S:modularVOAs}

\subsection{Spaces of formal series}
\label{ss:formalseries}
Let $V$ a $\K$-vector space, and $V[z]$ (resp. $V[[z]]$, resp. $V(\!(z)\!)$) the space of polynomials (resp. formal power series, resp. formal Laurent series) in $z$ with coefficients in $V$. We warn the reader that $V \otimes_\k \k[[z]] \subsetneq V[[z]]$ when $V$ is infinite dimensional. These spaces are equipped with a higher derivation $\partial_z = (\partial_z^{(i)})_{\ge 0}$ where $\partial_z^{(0)}$ is the identity map and $\partial_z^{(k)}$ is defined by
\begin{eqnarray}
\label{e:dividedpowerders}
\partial^{(i)}z^{m}=\binom{m}{i}z^{m-i} \text{ for } i\in \N, \,\, m\in \mathbb{Z} .
\end{eqnarray}
The reader can easily check that \eqref{e:HSrel} holds. 
Note that if $i=i_0+i_1p+\cdots + i_kp^{k}$ is the $p$-adic expansion of $i \in \Z_{>0}$ we have a factorisation
\begin{eqnarray}
\label{e:padicHSdecomp}
\partial^{(i)}=(\partial)^{i_0} (\partial^{(p)})^{i_1}\cdots (\partial^{(p^k)})^{i_k}.
\end{eqnarray}

\subsection{Modular vertex algebras }

Vertex algebras have been widely studied over the complex numbers, and so to emphasise the fact that our underlying field has positive characteristic we call them {\it modular vertex algebras}, see \cite{Bo1, BR}. 

\begin{definition}\label{d2.1} A modular vertex algebra is a $\K$-vector space $V$ equipped with a family of bilinear products 
\begin{equation*}
\begin{split}
&V\otimes V\rightarrow V\\
&a\otimes b \rightarrow a_{(n)}b\, ,  \quad n\in \ZZ\, , 
\end{split}
\end{equation*}
and a vector $\vac\in V$ such that
\begin{enumerate}
\item $a_{(n)}b=0$ for $n \gg 0$, 
\item $\vac_{(n)}a=\delta_{n,-1}a$, 
\item $a_{(-1)}\vac=a$, 
\item  The following identity is satisfied
\begin{equation}\label{B}
\sum_{j\in \ZZ_+} (-1)^{j}\binom{n}{j}\left(a_{(m+n-j)}b_{(k+j)}-(-1)^{n}b_{(n+k-j)}a_{(m+j)})\right) 
=\sum_{j\in \ZZ_+} \binom{m}{j}( a_{(n+j)}b)_{(m+k-j)} \, .
\end{equation}
\end{enumerate}

\end{definition}

We also have the following equivalent definition.

\begin{definition}\label{d2.3}  
A modular vertex algebra is a $\K$-vector space $V$, equipped with a vector $\vac \in V$, a family of endomorphism $T^{(k)}\in \End(V)$, $k\in \ZZ_{> 0}$,  an a linear map 
\[Y\colon  V\to \Hom_{\k}(V, V(\!(z)\!))\,, \qquad a\mapsto Y(a,z)\,, \]
subject to the following axioms:

\begin{enumerate}
\item[(1$^\prime$)] \; $Y(\vac,z)=I$, \; $Y(a,z)\vac\in V[[z]]$, \; $Y(a,z)\vac\big|_{z=0} = a$.

\item[(2$^\prime$)] \; 
$Y(T^{(k)}a,z)=\partial^{(k)}_z Y(a,z)$\ \  and \ \  $T^{(k)}\vac=0$.

\item[(3$^\prime$)] \;  For every $a,b\in V$, there exists $N\in\N$ such that 
\begin{equation*}
\begin{split}
(z_{1}-z_2 )^{N}Y(a,z_1)Y(b,z_2) = (z_{1}-z_2)^{N}Y(b,z_2) Y(a, z_1) \, .\\
\end{split}
\end{equation*}
\end{enumerate}
\end{definition}
The equivalence between the Definition \ref{d2.1} and Definition \ref{d2.3} is well-known over $\C$ and is explained in \cite[Theorem~4.3]{Ma} in case the base is an arbitrary ring. We sketch the correspondence here. From Definition \ref{d2.1} we define $T^{(k)}\in \End(V)$ for $k\geq 1$ and $Y:V\rightarrow\End(V)((z))$ as follows
\begin{equation*}
\begin{split}
& Y(a,z)=\sum_{n\in \ZZ}a_{(n)}z^{-n-1}\, , \qquad T^{(k)}(a):=a_{(-k-1)}\vac \, , \qquad \text { for } a\in V.
\end{split}
\end{equation*}
The conditions (1), (2) and (3) follow easily from \eqref{B}. On the other hand, using Definition \ref{d2.3} we can write $Y(a,z) = \sum_{n\in \Z} a_{(n)} z^{-n-1}$ and the operators $a_{(i)}$ satisfy the required axioms.

\begin{definition}
For a vertex algebra $V$ the centre is the subspace  
\begin{equation*}\label{eq1.3}
\mathfrak{z}(V):=\{b\in V \mid a_{(n)}b=0 \text{ for all } a\in V \text{ and } n\geq 0\}\, .
\end{equation*}
\end{definition}
From \eqref{B} with $n=0$ we see that
the centre coincides with the set of elements $v\in V$ such that 
\[[v_{(m)}, a_{(k)}]=0 \qquad \text{for all }a\in V\]
for all $m,k\in \ZZ$. 

\subsection{Modular Poisson vertex algebras}
\label{ss:modularPVA}
A commutative vertex algebra $V$ is one such that $u_{(n)} v = 0$ for all $u, v \in V$ and $n \ge 0$. Equivalently it satisfies $Y( \bullet, z) : V \to V[[z]]$. In this case, the product  $a, b \mapsto a_{(-1)}b$ is commutative, associative and unital with unit $\vac$. And, the endomorphisms $T^{(k)}$ define a higher derivation of this commutative algebra \eqref{e:HSrel}. Moreover the category of commutative differential algebras is equivalent to the category of commutative vertex algebras. The key difference between characteristic zero and the modular setting is that differential algebras over $\C$ are equipped with a higher differential of order 1, as opposed to a higher differential over $\k$. 

\begin{definition} 
Let $(V, \vac, T)$ be a commutative vertex algebra, with $T = (T^{(k)})_{k \ge 0}$. 

For a formal power series $a(z) = \sum_{m\in \Z}a_m z^m$ the {\it polar part} of  $a(z)$ is defined to be $a_-(z) =\sum_{m<0} a_{(m)}z^m$.

We say that $V$ is a Poisson vertex algebra if there is a linear map
\begin{eqnarray}
\begin{array}{rcl}
Y_{-} & : & V\to \Hom(V, z^{-1}V[z^{-1}]),\\
& & a\mapsto Y_{-}(a,z)=\sum_{n\geq 0}a_{(n)}z^{-n-1}
\end{array}
\end{eqnarray}
subject to the following axioms:

\begin{enumerate}
\item[(i)]  $Y_{-}(T^{(k)}a,z)=\partial^{(k)}_z Y(a,z)$.

\item[(ii)] \; 
$Y_{-}(a,z)b=\left(\sum_{k\geq 0} z^kT^{(k)} Y(b,-z)a\right)_{-}$.

\item[(iii)] \;  $[a_{(m)}, Y_{-}(b,w)]=\sum_{k\geq 0}\binom{m}{k}\left(w^{m-k}Y_{-}(a_{(k)}b, w)\right)_{-}\, .$ 

\item[(iv)] \;  $Y_{-}(a,w)(b\cdot c)=(Y_{-}(a,w)b)\cdot c+(Y_{-}(a,w)c)\cdot b$ \, .
\end{enumerate}
We refer the reader to \cite[\textsection 16.2]{FB} for an introduction.
\end{definition}

Now suppose that $V$ is a graded vertex algebra $V = \bigoplus_{i \ge 0} V_i$. Then we can define a filtration $V = \bigcup_{i \ge 0} F_i V$ by letting $F_d V$ be the subspace spanned by
\begin{eqnarray}
\label{e:weightdepend}
a_{-n_1 - 1}^1 a_{-n_2 - 1}^2 \cdots a_{-n_m - 1}^m \vac
\end{eqnarray}
where $a^1,...,a^m \in V$ are strong generators with $a^i \in V_{d_i}$ and
$\sum_{i=1}^m d_i \le  d.$ For the following theorem see for example \cite[\textsection 3.7]{Ar}. 

\begin{theorem}
This filtration is compatible with the vertex algebra structure, and the associated graded algebra $\gr V$ inherits the structure of a Poisson vertex algebra.
\end{theorem}

\subsection{Modular affine vertex algebras} 

Let $G$ be a connected, simple group scheme over $\k$ such that the characteristic of $\k$ is very good for the root system. Recall that $\kappa$ denotes the normalised Killing form \eqref{e:normalisedKilling} on $\g = \Lie(G)$.

The loop Lie algebra is $\g(\!(t)\!) := \g \otimes \k(\!(t)\!)$ and the Kac--Moody affinization $\hg := \g(\!(t)\!) \oplus \k K$ of $\g$ is the central extension of $\g(\!(t)\!)$ afforded by $\kappa$. Note that $K \in \hg$ is central and the Lie brackets in $\hg$ are determined by $[xt^i, yt^j] = [x,y]t^{i+j} + \delta_{i, -j} \kappa(x,y) i K$. We use the standard notation $xt^i := x_i$ which may denote either an element or an endomorphism of $V^k(\g)$. 

Let $\k_k$ be the one dimensional representation of $\g[[t]] \oplus \k K$ upon which $\g[[t]]$ acts trivially and $K$ acts by $k$. The vacuum module for $\hg$ is the induced module
\begin{equation}\label{eq1.2}
V^{k}(\g) := U(\widehat{\lieg})\otimes_{\g[[t]] \oplus \k K}\k_{k},
\end{equation}
and $k$ is known as {\it the level} of $V^k(\g)$.


The vacuum module carries a canonical structure of modular vertex algebra, see \cite{FB, Ka}.  
If we pick and ordered basis $\{x^i \mid  i=1, ..., \dim \mathfrak{g}\}$ for $\g$, then the Poincare--Birkhoff--Witt (PBW) theorem implies that we have a basis given by ordered monomials
 \begin{equation}\label{pbw}
 x^{i_1}_{-n_1-1}\cdots x^{i_{m}}_{-n_m-1}\vac
 \end{equation}
  such that $0 \le n_1\leq n_2\leq \cdots \leq n_{m}$, and $i_{j}\leq i_{j+1}$ whenever $n_{j}=n_{j+1}$. The vacuum vector in $V^{k}(\mathfrak{g})$ is $\vac=1\otimes 1$. For $k\in \N$ the translation operator $T^{(k)}$ is determined by $T^{(k)}\vac=0$ and  
\begin{eqnarray}
\label{e:trans1}
T^{(k)}x^{i}_{-n-1}\vac=\binom{n+k}{k}x^{i}_{-n-k-1}\vac 
\end{eqnarray}
 whilst in general, $T^{(k)}$ acts on $V^k(\g)$ by
\begin{eqnarray}
\label{e:trans2}
T^{(k)}(  x^{i_1}_{n_1}\cdots x^{i_{m}}_{n_m}\vac)=\sum_{j_1+\cdots + j_m=m} T^{(j_1)}(x^{i_1}_{n_1})\cdots T^{(j_m)}(x^{a_{m}}_{n_m})\vac\, .
\end{eqnarray}
The fields are defined as follows. We have $Y(\vac,z)=\Id_{V^k(\g)}$ and
\[Y(x^{i}_{-1}\vac,z)=x^i(z) :=\sum_{n\in \ZZ} x^{i}_{n}z^{-n-1},\]
 whilst in general we have
\begin{equation}\label{no}
Y(  x^{i_1}_{-n_1-1}\cdots x^{i_{m}}_{-n_m-1}\vac,z)=:\partial_{z}^{(n_1)}x^{i_1}(z)\cdots \partial_{z}^{(n_m)}x^{i_m}(z):
\end{equation}
where $:\ \  :$ denotes the normally ordered product (see \cite{Ka}).  The locality of the fields is a consequence of the explicit description of Lie brackets in $\hg$.

In this article we are interested in describing the centre $\mathfrak{z}(V^{k}(\lieg))$. Consider the space of $\g[ [t]]$ invariants
\[ V^{k}(\lieg)^{\g[[t]]} =\{b\in V^{\kappa}(\lieg) \mid (x_n)\/b=0 \text{ for } n \ge 0\}\, .\]
The following useful property follows from \eqref{no}; see \cite[Lemma~3.1.1]{Fr} for more detail.
\begin{lemma}\label{lin}
$\mathfrak{z}V^{k}(\g)= V^{k}(\lieg)^{\lieg[[t]]}.$ $\hfill \qed$
\end{lemma}

\subsection{The $p$-centre of modular affine vertex algebras}

We introduce a central subalgebra of $V^{k}(\lieg)$, following \cite{AW}. The restricted structure on $\g$ gives rise to a natural restricted structure on $\hg$ determined by $(xt^i)^{[p]} = x^{[p]} t^{pi}$ and $K^{[p]} = K$. We denote the $p$-centre of $U(\hg)$ by $Z_p(\hg)$, which is generated by $\{x^p - x^{[p]} \mid x\in \hg\}$. It is isomorphic to a polynomial ring in countably many generators.

We define the {\it $p$-centre $\z_p V^k(\g) \subseteq V^k(\g)$} to be the image of the natural map $Z_p(\hg) \to U(\hg) \to V^k(\g)$. Note that $\mathfrak{z}_{p} V^{k}(\lieg)\subseteq V^{k}(\lieg)^{\lieg[[t]]}$
, and so Lemma \ref{lin} implies the following result.

\begin{lemma}
\label{L:pcentre}
$\mathfrak{z}_{p} V^{\kappa}(\lieg)\subseteq \mathfrak{z}V^{\kappa}(\lieg).$ $\hfill \qed$
\end{lemma}

\subsection{The current group and the centre} 
\label{ss:Ginvandcentre}

The {\it current group} $G[[t]]$ of $\k[[t]]$-points of $G$ is equal to the group of $\k$-points of $J_\infty G$, by definition. Let $\epsilon$ be a formal variable satisfying $\epsilon^2 = 0$. The Lie algebra $\Lie G[[t]]$ of $G[[t]]$ is defined to be the kernel of the homomorphism $G(\k[[t]][\epsilon]) \to G(\k[[t]])$ coming from the map $\k[[t]][\epsilon] \to \k[[t]]$ by functoriality. By \cite[p. 192]{Mi} we have $\Lie G[[t]] = \g[[t]]$.  See also \cite{AM}. 
\begin{lemma}
\label{L:currentaction}
For any $k\in \k$ there is a natural rational $G[[t]]$-action on $V^{k}(\g)$ such that the differential at the identity element is the left action of $\g[[t]]$.
\end{lemma}
\begin{proof}

The action of $G[[t]]$ on $\hat{\g}\cong\g((t))\oplus \k K$ is given by $g\cdot (A(t)+cK)=gA(t)g^{-1}+K\Res_{t}(g^{-1}dg,y)+K c$ where $g^{-1}dg$ is the logarithmic differential, see also \cite[Sec 1.3.6]{Fr}. This action lift to an  action on $U({\hg})$ and on the vertex algebra through the surjective map $U(\hg) \to V^{k}(\g)$.

\end{proof}

The next result follows from Lemma~\ref{lin} and Lemma~\ref{L:currentaction}.
\begin{corollary}
\label{C:groupinvariantscentre}
$V^k(\g)^{G[[t]]} \subseteq \z V^k(\g)$. $\hfill \qed$
\end{corollary}

\section{Invariants and Arc spaces}

\subsection{The semiclassical limit of the affine vertex algebra}
We equip $V^k(\g)$ with the filtration described in \eqref{e:weightdepend}. This filtration is ascending, exhaustive $V^k(\g) = \bigcup_{d \ge 0} F_d V^k(\g)$, and connected $F_{0}V^{k}(\mathfrak{g})=\KK\vac$. Furthermore, $G[[t]]$, $\g[[t]]$ and $\{T^{(i)} \mid i \ge 0\}$ all preserve the filtered pieces.

This gives rise to an associated graded space
\begin{align}
\label{e:gradedV}
\begin{array}{rcl}
\gr V^{k}(\mathfrak{g}) & =&\bigoplus_{d\geq 0}F_dV^{k}(\mathfrak{g})/F_{d-1}V^{k}(\mathfrak{g})\vspace{5pt}\\
&=& S(\g[t^{-1}] t^{-1})=\k[J_\infty \g^*]
\end{array}
\end{align}
where the final identification follows from the remarks of Section~\ref{ss:arcsandjets} (see \cite[\textsection 3.8, (34)]{Ar} for more detail). There is an induced  action of $G[[t]]$ and $\g[[t]]$ on $\gr V^{k}(\g)$. Via \eqref{e:gradedV} these actions coincide with the natural ones on $S(\hg) / (\g_+^0)$ which arise from the adjoint action of $G[[t]]$ on $\g(\!(t)\!) / \g[[t]]$, in the same manner as Lemma~\ref{L:currentaction}. This is an isomorphism of $\g[[t]]$-modules, and we will identify $\gr V^k(\g) = \k[J_\infty \g^*]$ in the sequel.

Similarly the translation operator $T$ gives rise to a higher derivation $\partial = (\partial^{(0)}, \partial^{(1)}, \partial^{(2)},...)$ on $S(\g[t^{-1}] t^{-1})$ by the rule $\partial^{(i)}( v + F^{d-1}V^k(\g)) := T^{(i)}(v) + F^{d-1} V^k(\g)$ for $i, d\ge 0$. More explicitly, $\partial^{(i)}$ coincides with $(-1)^i\partial_t^{(i)}$ where $\partial_t$ is described in Section~\ref{ss:formalseries}.

Combining our remarks with Lemma~\ref{lin} we have the inclusion
\begin{eqnarray}
\label{e:lineq}
\gr \z V^k(\g) \subseteq \k[J_\infty \g^*]^{\g[[t]]}.
\end{eqnarray}
We also have
\begin{eqnarray}
\label{e:peq}
\gr \z_p V^k(\g) = S(\g[t^{-1}]t^{-1})^p.
\end{eqnarray}

By Theorem \ref{th1} we have a homogeneous invariant generating set $\bP_1,...,\bP_r \in \k[\g^*]^G$. If we let $G[[t]]$ act on $\k[\g^*]$ via the homomorphism $G[[t]] \onto G$ then the inclusion $\k[\g^*] \subseteq \k[J_\infty \g^*]$ given by $x \mapsto x_{-1}$ is $G[[t]]$-equivariant. Therefore the elements $\{\bP_{i,-1} \mid  i=1,...,r\}$ we obtain are $G[[t]]$-invariant.
 
For $j > 0$ we define regular functions $\bP_{i,-j} \in \k[J_\infty \g]$ by
\begin{equation}
\label{eqp2}
\bP_{i,-j}:=\partial^{(j-1)}\bP_{i,-1}.
\end{equation}

We can define these regular functions on $J_\infty \g^*$ equivalently using fields. For $i=1,...,\dim \g$ we let $x^i(z)=\sum_{n < 0}x^{i}_{n}z^{-n-1}$. The change of variables $x^i \mapsto x^i(z)$ produces a map $\k[\g^*] \mapsto \k[J_\infty \g^*][z]$ and the image of $\bP_i$ under this map is the generating function for the series $(\bP_{i,-n-1})_{n \ge 0}$.

In the following lemma we write $x^1,...,x^n$ for a basis of $\g$.
  \begin{lemma}\label{L:rewriteders} 
For $P \in \k[\g^*]$ consider $P_{-1} \in \k[J_\infty \g^*]$ as above. Then
  $$\frac{\partial}{\partial x^{i}_{-1-s}} \partial^{(m)} P_{-1} =\begin{cases} \partial^{(m-s)}\frac{\partial P_{-1}}{\partial x^{i}_{-1}}, \qquad m\geq s\geq 0\, , \\
0 \, ,  \qquad\qquad\qquad\quad m< s\,  .    \end{cases} $$

 \end{lemma}
 \begin{proof}
It suffices to check the claim for monomials in $x^1,...,x^n$. The case $m < s$ is clear because $\partial^{(m)} P_{-1}$ only depends on variables $x^j_{-1},...,x^j_{-1-m}$ with $j=1,...,\dim\g$.

We prove the statement for $m \ge s$ by induction on the total degree of $P$. The degree 1 case is easily verified. Now suppose it has been verified for all $P$ of degree less than $d$. Any monomial of degree greater than 1 can be written as a product $PQ$ of monomials of degrees strictly less than $d$.

Using \eqref{e:HSrel} have
\begin{eqnarray*}
\partial^{(m-s)} \frac{\partial}{\partial x^i_{-1}}(P_{-1} Q_{-1}) & = &
\sum_{j=0}^{m-s} \big( (\partial^{(j)} \frac{\partial P_{-1}}{\partial x^i_{-1}})(\partial^{(m-s-j)} Q_{-1}) + (\partial^{(j)} P_{-1})(\partial^{(m-s-j)} \frac{\partial Q_{-1}}{\partial x^i_{-1}})\big)\\
& = &\sum_{j=s}^m (\partial^{(j-s)} \frac{\partial P_{-1}}{\partial x^i_{-1}})(\partial^{(m-j)} Q_{-1)}) + \sum_{j=0}^{m-s} (\partial^{(j)} P_{-1})(\partial^{(m-s-j)} \frac{\partial Q_{-1}}{\partial x^i_{-1}}).
\end{eqnarray*}
In the second line we shift the indexes for half of the summands.

On the other hand we have
\begin{eqnarray*}
\frac{\partial}{\partial x^i_{-1-s}} \partial^{(m)}(P_{-1} Q_{-1}) & = &
\sum_{j=0}^m \big( (\frac{\partial}{\partial x^i_{-1-s}} \partial^{(j)} P_{-1}) (\partial^{(m-j)} Q_{-1}) + (\partial^{(j)} P_{-1})(\frac{\partial}{\partial x^i_{-1-s}} \partial^{(m-j)} Q_{-1})\\
& =& \sum_{j=s}^m (\frac{\partial}{\partial x^i_{-1-s}}\partial^{(j)} P_{-1})(\partial^{(m-j)} Q_{-1}) + \sum_{j=0}^{m-s} (\partial^{(j)}P_{-1}) (\frac{\partial}{\partial x^i_{-1-s}} \partial^{(m-j)} Q_{-1})
\end{eqnarray*}
where the second line uses the vanishing statement of the lemma for $s<m$. Since the degree of both $P$ and $Q$ is less than the degree of $PQ$ the proof concludes using the inductive hypothesis.
\end{proof}

Over the complex numbers the $G[[t]]$-invariance of the polynomials $P_{i,-j}$ is equivalent to the $\g[[t]]$-invariance. In our setting this is no longer the case, and so we check the invariance.

\begin{lemma}
\label{L:PVSinvariants}
$\bP_{i,-j} \in \k[J_\infty \g^*]^{G[[t]]}$.
\end{lemma}

\begin{proof}
The coadjoint action $G\times \g^{*}\rightarrow \g^{*}$ is given by $\mu:\k[\lieg^{*}]\rightarrow \k[\lieg^{*}]\otimes \k[G]$, we define $\iota:\k[\lieg^{*}]\rightarrow \k[\lieg^{*}]\otimes \k[G]$ by $\iota(f)=f\otimes 1$ for $f\in \k[\lieg^{*}]$. Then, $\k[\g^{*}]^{G}=\Ker (\mu-\iota)$. Moreover, $\k[J_{\infty}\g^{*}]^{J_{\infty}G}=\Ker (J_{\infty}\mu-J_{\infty}\iota)$ and $\bP_{i,-j-1}=\partial^{(j)}P_{i,-1}$ then $(J_{\infty}\mu-J_{\infty}\iota)\bP_{i,-j-1}=\partial^{(j)} (J_{\infty}\mu-J_{\infty}\iota)\bP_{i,-1}=0$. Note that $J_\infty\mu$ and $J_\infty \iota$ commute with $\partial$ since $J_\infty$ is a functor to differential algebras.

\end{proof}

We provide one example to show that the invariance of $P_{i,-j}$ can be derived directly in special cases.
\begin{example}
Let $G = \SL_{n+1}$ so that
\[G[[t]]=\{g\in \operatorname{Mat}_{n+1}(\k[[t]]))  \mid \det (g)=1\} \]
Let $\{x^{i,j} \mid 1\le i,j\le n+1\}$ be the standard basis for $\operatorname{Mat}_{n+1}(\k)$. Define a matrix $X \in \operatorname{Mat}_{n+1}(\k[t^{-1}]t^{-1}[[z]])$ whose $i,j$ entry is the series $x^{i,j}(z) = \sum_{m < 0} x^{i,j}_m z^{-1-m}$. Write
$$\det(\lambda - X) = \lambda^n + \sum_{i=0}^n P_{i}(z) \lambda^i$$
We identify $\operatorname{Mat}_{n+1}(\k[t^{-1}]t^{-1})$ with $\Hom_\k(\operatorname{Mat}_{n+1}(\k[[t]]), \k)$ via the residue pairing. Then the $z^{j}$ coefficient of the series $P_i(z)$ restricts to $P_{i,-1-j}$ on $J_\infty \sl_{n+1}\subseteq \operatorname{Mat}_{n+1}(\k[[t]])$.

The group $G[[t]]$ acts on $\operatorname{Mat}_{n+1}(\k[t^{-1}]t^{-1}) = J_\infty \gl_{n+1}^*$ and the invariance of the $P_{i,-j}$ is now a simple consequence of the fact that the characteristic polynomial is $G[[t]]$-invariant.
\end{example}

Finally, we prove the main theorem of this section, which is a vertex algebra analogue of Theorem~\ref{th1}(1)-(4).

\begin{theorem}
\label{T:PVAcentre}
\begin{enumerate}
\setlength{\itemsep}{4pt}
\item $\k[J_\infty \g^*]^{G[t]]} = \k[P_{i,-j} \mid i=1,..,r, \ j < 0]$ polynomial ring on infinitely many generators.
\item $\k[J_\infty \g^*]^{\g[[t]]}$ is generated by $\k[J_\infty \g^*]^{G[[t]]}$ and $\k[J_\infty \g^*]^p$.
\item $\k[J_\infty \g^*]^{\g[[t]]}$ is a free $\k[J_\infty \g^*]^p$-module with basis
\begin{eqnarray}
\label{e:jetsrestrictedbasis}
\Big\{\prod_{i,j} P_{i,-j}^{k_{i,j}} \mid 0\le k_{i,j} < p, \text{ finitely many nonzero}\Big\}.
\end{eqnarray}
\item $\k[J_\infty \g^*]^{\g[[t]]}$ is isomorphic to the tensor product of $\k[J_\infty \g^*]^p$ and $\k[J_\infty \g^*]^{G[[t]]}$ over the intersection.
\end{enumerate}
\end{theorem}

\begin{proof} 
In the proof we make use of the group of $m$-jets $J_m G$ which has Lie algebra $\g_m := \g \otimes \k[t]/(t^m+1)$ and $\k$-points $G_m$, known as the truncated current group.

The map $\g^*[[t]] \onto \g_m^*$ gives rise to an inclusion $\k[\g_m^*] \to \k[J_\infty \g^*]$. 
This subalgebra is a $G[[t]]$-submodule and the action factors through the homomorphism $G[[t]] \onto G_m$ arising from the inclusion of Hopf algebras $J_m\k[G] \into J_\infty \k[G]$ (see \eqref{e:JHopfstructure}). Now in order to prove the theorem it suffices to prove the corresponding claims about $\k[J_m\g^*]^{G_m}$ and $\k[J_m\g^*]^{\g_m}$, mutatis mutandis.

In order to prove (2) and (3), we apply a result of Skryabin \cite[Theorem~5.4]{S}. To satisfy the conditions of his theorem we must demonstrate that:
\begin{enumerate}
\item[i)]  $\ind (\lieg_{m}) =(m+1) \rank(\g)$
\item[ii)]  $\codim_{\lieg_{m}^*} (\lieg_{m}^*\setminus J(P_{i,-j}))\geq  2$
\end{enumerate}
where $\ind(\g_m) := \min_{\chi \in \g^*} \dim \g_m^\chi$ is the index of the Lie algebra, and $J(P_{i,-j})$ denotes the Jacobian locus of the $P_{i,-j}$, i.e. the open set of $\g_m^*$ consisting of points $\chi$ such that the functions $d_\chi P_{i,-j}$ are linearly independent.

Part (i) is a special case of a theorem of Ra{\"i}s--Tauvel (where they work over the complex numbers) and we include a short proof valid in our setting. Pick a triangular decomposition $\g = \n^- \oplus \h \oplus\n^+$, which gives $\g_m = \n^-_m \oplus \h_m \oplus \n_m^+$. Let $\h_m^{\circ} = \{\sum_{i=0}^{m} h_i t^i \mid h_0 \in \h^{\reg}\}$.

The differential of the adjoint action $G_m \times \h_m \to \g_m$ at a point $(1, h)$ is the linear map $\g_m \times \h_m \to \g_m$ given by $x,z \mapsto [x, h] + z$. When $h\in \h_m^\circ$ we have that 
\begin{eqnarray}
\label{e:oneforlater}
[h, \g_m] = \n_m^- \oplus \n_m^+ & \text{ and } & \g_m^h = \h_m\, , 
\end{eqnarray} 
in more detail, $[h_0, \mathfrak{n}^{\pm}]=\mathfrak{n}^{\pm}$ then $[h, \mathfrak{n}_m^{\pm}]=\mathfrak{n}_m^{\pm}$ and easily $[h, \mathfrak{h}_m]=0$. 
And so the differential is surjective for such a point. Consequently the conjugates of $\h_m^{\circ}$ are dense in $\g_m$. On the other hand $\dim \g_m^h = (m+1)\rank(\g)$ by \eqref{e:oneforlater}, and so by upper-semicontuinuity of dimensions of fibres of morphisms, we have that $\min_{x\in \g_m} \dim \g_m^x = (m+1) \rank(\g)$. 

There is a $\g_m$-invariant bilinear form on $\g_m$ given by $x t^i , yt^j \mapsto \delta_{i, m+1-j}\kappa(x, y)$, and this induces an isomorphism $\theta: \g_m \to  \g_m^*$ of $\g_m$-modules. We deduce that $\ind(\g_m) = (m+1) \rank(\g)$, which proves (i).

Let $\g_m^\circ = \{ \sum_{i=0}^m x_i t^i \mid x_0 \in \g^{\reg}\}$. Thanks to \cite[Proposition~3.2]{BG} we have $\codim_{\g}(\g \setminus \g^\circ) = 3$ hence we have $\codim_{\g_m}(\g_m \setminus \g_m^\circ) = 3$. We shall show that $\theta(\g_m^\circ) \subseteq J(P_{i,-j})$, which shall prove (ii). 

Once again we write $x^1,...,x^n$ for a basis of linear functions on $\g^*$, which gives a basis $\{x^i_{-j} \mid i=1,...,r, \ j=1,...,m+1\}$ of linear functions on $J_m \g^*$. Consider the Jacobian of the functions $\{P_{i,-j} \mid 1\le i \le r, \ 1 \le j \le m+1\}$, a matrix of size $(m+1)\dim\g \times (m+1) \rank\g$. It has block form
\begin{equation}\label{eq3.13}
\left(\begin{matrix} 
   \frac{\partial P_{i,-1}}{\partial x^{j}_{-1}}   &  &  & 0 \\
     \frac{\partial P_{i,-2}}{\partial x^{j}_{-1}}  & \frac{\partial P_{i,-2}}{\partial x^{j}_{-2}} &  \\
    \vdots & &\ddots  &  \\
     \frac{\partial P_{i,-1-m}}{\partial x^{j}_{-1}}  & \frac{\partial P_{i,-1-m}}{\partial x^{j}_{-2}} & \dots         & \frac{\partial P_{i,-m-1}}{\partial x^{j}_{-m-1}} 
    \end{matrix}\right)
    \end{equation}
where each block has size $(\dim \g)  \times (\rank \g)$. Using Lemma~\ref{L:rewriteders} we see that the diagonal blocks are all equal to $(\frac{\partial P_{i,-1}}{{\partial x^{j}_{-1}}})$. Evaluating at any point of $\g_m^\circ$ we see that this matrix has full rank, thanks to the differential criterion for regularity of $\g$ (see Lemma~3.4 and Proposition~3.4 of \cite{BG}). We have now shown that $\theta(\g_m^\circ) \subseteq J(P_{i,-j})$, which completes the proof of (ii), and part (2) and (3) of the current theorem.
    
For part (1)  it suffices to show that $\k[J_m\g^*]^{G_m} = \k[P_{i,-j} \mid i=1,...,r, \ -m-1 \le -j < 0]$ is a polynomial ring. Suppose that $f\in \k[J_m\g^*]^{G_m}$. Then by part (1) we can write $f$ uniquely as $f = \sum_{b\in B} g_b b$ where $B$ denotes the collection of monomials described in \eqref{e:jetsrestrictedbasis} (with all $j \le m+1$) and $g_b \in \k[J_m\g^*]^p$. By uniqueness, and using Lemma~\ref{L:PVSinvariants}, we see that each $g_b$ is $G_m$-invariant. Since $\k[J_m\g^*]$ has no nilpotent elements it follows that $g_b = (\bar g_b)^p$ for some $\bar g_b \in \k[J_m\g^*]^{G_m}$. Now part (1) may be proven by a downward induction on total degree.

Part (4) follows from \cite[Lemma~2.1]{GT}, whose proof does not require the free basis to be finite.
\end{proof}

\section{Centre of $V^{k}(\g)$ at the critical level}
\label{S:criticalcentre}

\subsection{Integral forms on vertex algebras}
\label{ss:integralforms}

Let $G_\Z$ be a connected, split simple group such that the base change of $G_\Z$ to $\k$ is equal to $G$. Also let $T_\Z$ be a choice of split maximal torus whose base change to $\k$ is $T$.

Let $R$ be the subring of $\Q$ generated by the reciprocals of integers between 1 and $h$. Since $p > h$ there is a natural homomorphism $R \to \k$. Write $G_\C$ for the base change to $\C$, write $\g_\C = \Lie G_\C$. Similarly we write $G_R$ and $\g_R$ for the $R$-defined objects. Since $\g$ is simple for $p > h$ we can identify $\g_\Z \otimes_\Z \k = \g$.

Since $S(\g_R[t^{-1}]t^{-1})$ is a free $R$-module, it acts as an intermediary between the symmetric algebras $S(\g_\C[t^{-1}] t^{-1})$ and $ S(\g[t^{-1}]t^{-1})$. to be precise we have natural identifications
\begin{eqnarray}
S(\g_R[t^{-1}] t^{-1}) \otimes_R \C & = & S(\g_\C[t^{-1}] t^{-1}) \\
S(\g_R[t^{-1}]t^{-1})  \otimes_R \k & = & S(\g[t^{-1}]t^{-1}).
\end{eqnarray}

Now we address the modular reduction of the affine vertex algebra $V^{-h^\vee}(\g_\C)$. We let $\hg_R = \g_R \otimes_R R[t^{\pm 1}] \oplus R K$, which is a Lie $R$-subalgebra of $\hg_\C$, and a free $R$-module. The enveloping algebra $U(\hg_R)$ satisfies the conditions of the PBW theorem, and we can construct the vacuum module $V^{-h^\vee}(\g_R) := U(\hg_R) \otimes_{\hg_+, R} R_{-h^\vee}$ (Cf. \eqref{eq1.2}), where $\hg_{+,R}$ is the induced $R$-form on $\hg_+$ and $R_{-h^\vee}$ is the rank one free $R$-module with $\hg_{+,R}$-action given by letting $K$ act via $-h^\vee$ and $\g_R[[t]]$ acts trivially. This is a free $R$-module and so we have
\begin{eqnarray}
  V^{-h^\vee}(\g_R) \otimes_R \C& = & V^{-h^\vee}(\g_\C) \\
  V^{-h^\vee}(\g_R) \otimes_R \k & = & V^{-h^\vee}(\g).
\end{eqnarray}

\subsection{A criterion for the existence of Segal--Sugawara vectors}
\label{ss:criterion}

Recall that $R$ is the subring of $\Q$ generated by $\{1/n \mid 1\le n \le h\}$, and now let $Q \subseteq \Q$ be a finitely generated $R$-algebra. Retain the notation $G_\C, \g_\C, G_R, \g_R$ from Section~\ref{ss:integralforms}. In this section we discuss the centre of $V^{-h^\vee}(\g)$ at the critical level, and so we ease notation further by writing $V_A = V^{-h^\vee}(\g_A)$ whenever $A \in \{R, Q, \C, \k\}$. Similarly we write $S_A = S(\g_A[t^{-1}]t^{-1})$.

Recall that $\z V_\k$ is a commutative vertex algebra and, in particular, it is a commutative (differential) algebra.  To prove Theorem~\ref{T:maintheorem} we use the following criterion, which follows directly from Theorem~\ref{T:PVAcentre} by modular reduction and a standard filtration argument.

\begin{proposition}
\label{P:AWcriterion}
Suppose that there is a non-zero homomorphism $Q \to \k$ and that there exist elements $S_{i,-1} \in V_Q \cap \z V_\C$ such that
\begin{eqnarray}
\label{e:property}
\gr S_{i,-1} = P_{i,-1}
\end{eqnarray}
for $i=1,...,r$ . If $S_{i,-j} := T^{(j-1)} S_{i,-1}$ and denote the image of $S_{i,-j}$ under $V_Q \to V_\k$ by the same symbol. The following hold:
\begin{enumerate}
\setlength{\itemsep}{2pt}
\item $V_\k^{G[[t]]} = \k[S_{i,-j} \mid i=1,..,r, \ j > 0]$ is a polynomial algebra on infinitely many generators.
\item $\z V_\k$ is generated by $V_\k^{G[[t]]}$ and $\z_p V_\k$ as a commutative algebra.
\item $\z V_\k$ is a free $\z_p V_\k$-module of over $\z_p V_\k$ with basis
\begin{eqnarray}
\label{e:bigbasis}
\Big\{\prod_{i,j} S_{i,-j}^{k_{i,j}} \mid 0\le k_{i,j} < p, \text{ finitely many nonzero}\Big\}.
\end{eqnarray}
\item $\z V_\k$ is isomorphic to the tensor product of $\z_p V_\k$ and $V_\k^{G[t]]}$ over their intersection.
\end{enumerate}
\end{proposition}

\begin{proof}
Recall Lemma~\ref{lin}, Lemma~\ref{L:pcentre} and Corollary~\ref{C:groupinvariantscentre} which state that $\z_p V_\k, V_\k^{G[[t]]} \subseteq \z V_\k = V_\k^{\g[[t]]}$.

It is clear that $V_Q$ is a vertex $Q$-subalgebra of $V_\C$ and that the translation operators on $V_Q$ and $V_\k$ intertwine the natural map $V_Q \to V_\k$. Therefore $\gr S_{i, -j} = P_{i,-j}$ for all $i, j$. Since $\C$ has characteristic zero, $\g_\C[[t]]$-invariants are $G_\C[[t]]$-invariant, and so we have $S_{i,-j} \in V_\C^{G_\C[[t]]}$.

Now let $T$ be either $\C, Q$ or $\k$. By Lemma~\ref{L:Gschemegenerators} we know that $J_\infty G_T$ is generated by $\{J_\infty u_\alpha \mid \alpha\in \Phi\}$. The group of $T$-points of $J_\infty u_\alpha$ is isomorphic to the additive group $T[[t]]$. For each $m \ge 0$ and $\alpha \in \Phi$ we consider the morphism of schemes
$u_{\alpha, m}^T : \mathbb{G}_a \to J_\infty G_T$ determined by setting $u_{\alpha, m}^T(a) = J_\infty u_\alpha (a t^m)$, for any $T$-algebra $A$ and element $a\in A$. Thus we have a group homomorphism $u_{\alpha, m}^T(T) : T \to G_T[[t]]$ of $T$-points. If $f \in V_T$ then $u_{\alpha, m}(T)(z) \cdot f$ is a polynomial expression in $z$ for $z \in T$.

We now fix $i,j$ and $\alpha, m$. The $G_\C[[t]]$-invariance of $S_{i,-j}$ implies that $u_{\alpha, m}^\C(\C)(z) \cdot S_{i,-j}^\C - S_{i,-j}^\C$ vanishes identically for all values of $z\in \C$. We can regard this as a polynomial identity in $V_T[z]$ where $z$ is a formal variable. Reducing modulo $p$ we deduce that $u_{\alpha, m}^\k(\k)(z) \cdot S_{i,-j} - S_{i,-j} = 0$ in $V_T$. Since this holds for every $\alpha, m$, and since $G[[t]]$ is generated by the images of the morphisms $u_{\alpha, m}^\k(\k)$ we deduce that $S_{i,-j}$ is $G[[t]]$-invariant.

Using an argument identical to the proof of Lemma~\ref{L:PVSinvariants} we see that the images $S_{i,-j} \in V_\k$ are $G[[t]]$-invariant. Property \eqref{e:property} continues to hold in $V_\k$ and so we have $\k[S_{i,-j} \mid i,j] \subseteq V_\k^{G[[t]]}$. By Theorem~\ref{T:PVAcentre} we have $\gr \k[S_{i,-j} \mid i,j] \subseteq \gr V_\k^{G[[t]]} \subseteq S_\k^{G[[t]]} = \k[P_{i,-j} \mid i,j] = \gr \k[S_{i,-j} \mid i,j]$ and so we must have equality throughout. This proves (1).

Part (2) follows by a very similar line of reasoning, in view \eqref{e:lineq} and \eqref{e:peq}. Arguing in the same fashion once again, we see that the elements \eqref{e:bigbasis} span $\z V_\k$ as a $\z_p V_\k$-module. If there were a $\z_p V_\k$-linear relation amongst them the the projection into $\gr V_\k$ would give an $S_\k^p$-linear relation amongst $\{P_{i,-j} \mid i,j\}$, which contradicts Theorem~\ref{T:PVAcentre}(3). This proves (3).

Part (4) follows from \cite[Lemma~2.1]{GT}, whose proof does not require the free basis to be finite.
\end{proof}

In order to complete the proof of Theorem~\ref{T:maintheorem} we must show that the hypotheses of Proposition~\ref{P:AWcriterion} can be satisfied for each simple finite root datum.
 
\subsection{Completing the proof: type {\sf A}}
\label{ss:completingA}
Let $n \in \N$. The known formulas for generators of the Feigin--Frenkel centre in type {\sf A} are expressed inside the vertex algebra $V^{-h^\vee}(\gl_N(\C))$. The Coxeter number $h$ for $\SL_N$ is $N$. Since we have assumed that $p > h$ it follows that $\gl_N(R)$ is the direct sum (as Lie $R$-algebras) of $\sl_N(R)$ and the rank 1 centre $R I$.

The normalised Killing form on $\gl_N(R)$ is given by $\kappa(x,y) = \Tr(xy) - \frac{1}{N} \Tr(\ad(x) \ad(y))$. The kernel is the centre of $\gl_N(R)$ and the restriction to $\sl_N(R)$ is the normalised Killing form \eqref{e:normalisedKilling}.

Let $\cc_R = R I \otimes R((t)) \oplus R K \subseteq \widehat\gl_N(R)$ a commutative Lie subalgebra. Then the image $H_R$ of $U(\cc_R)$ in $V^{-h^\vee}(\gl_N(R))$ is a commutative vertex subalgebra, isomorphic to the semiclassical limit of the Heisenberg vertex algebra. Furthermore we have an $R$-linear isomorphism of vertex rings $V^{-h^\vee}(\gl_N(R)) \cong V^{-h^\vee}(\sl_N(R)) \otimes_R H_R$. Similar remarks hold over $\k$ or $\C$. The next lemma follows immediately from these remarks.

\begin{lemma}
Theorem~\ref{T:maintheorem} holds for $V^{-h^\vee}(\sl_N)$ if and only if it $V^{-h^\vee}(\gl_N)$. $\hfill\qed$
\end{lemma}

Generators of $V^{-h^\vee}(\gl_N(\C))$ satisfying the hypotheses of Proposition~\ref{P:AWcriterion} are described \cite[Theorem~3.1]{CM} or \cite[Theorem~7.1.3]{Mo}. They are manifestly defined over $R = \Z[1/n \mid 1\le n \le N]$ and this completes the proof in type {\sf A}. 

\subsection{Completing the proof: classical types}
\label{ss:completingBCD}
Let $G$ be a connected, simply connected, simple classical group scheme of type {\sf B}, {\sf C}, or {\sf D}. Under the hypothesis $p > h$ there are no inseparable isogenies, and so it does no harm to replace $G$ by one of the matrix groups $\SO_N$ or $\Sp_N$.

In this case the hypotheses of Proposition~\ref{P:AWcriterion} are satisfied with $Q = R$, thanks to Theorems~8.1.6, 8.1.9 and 8.3.8 of \cite{Mo}. See also \cite{MoSS} and \cite{Ya}.

\subsection{Completing the proof: exceptional types}

Finally we let $G$ be a simply connected, connected exceptional simple group scheme, and $\g= \Lie(G)$. Recall the notation from Section~\ref{ss:criterion}, and $x^1,...,x^n$ denotes a basis for $\g$.

Fix $i =1,...,r$. For $m > 0$ let $V_\C^{[m]}$ be the span of PBW monomials \eqref{pbw} of degree less than $\deg P_{i,-1}$ involving $x^j_{-k}$ where $m\ge k \ge 1$ and $j = 1,...,n$. 

Since $\g_\C[[t]]$ is topologically generated by $\g_\C$ and $\g_\C\otimes t$, the existence of an element $S_{i,-1} \in (V_\C^{[m]})^{\g_\C[[t]]}$ such that $\gr S_{i,-1} = P_{i,-1}$ can be expressed as a solution to a finite system of affine linear equations with integer coefficients. The Feigin--Fenkel theorem (see \cite[Theorem~8.1.3]{Fr} implies that there exists $m > 0$ such that the system has a solution and, by integrality, the existence of a solution in $(V_\C^{[m]})^{\g_\C[[t]]}$ implies the existence of a solution $S_{i,-1} \in (V_\Q^{[m]})^{\g_\Q[[t]]}$. Writing $S_{i,-1}$ as a $\Q$-linear span of elements of the PBW basis, we can choose an integer $q_i \in \N$ (say, the maximal denominator appearing in the coefficients) so that $S_{i,-1} \in V_{R_i}$ where $R_i = \Z[q_i^{-1}]$. Taking $q = \max\{q_i\! \mid i=1,...,r\}$ and setting $Q := R[q^{-1}]$ we have $S_{i,-1} \in V_Q \cap V_\C^{\g_\C[[t]]}$ for all $i=1,...,r$.

Now provided the characteristic of the field $\k$ is greater than $q$ there is a nonzero homomorphism $Q \to \k$, and so the hypotheses of Proposition~\ref{P:AWcriterion} are satisfied. This completes the proof of Theorem~\ref{T:maintheorem}.

   \bibliographystyle{amsalpha}

   \end{document}